\documentclass[letterpaper, reqno,11pt]{article}
\usepackage[margin=1.0in]{geometry}
\usepackage{color,latexsym,amsmath,amssymb, amsthm, graphicx,subfigure}
\usepackage[percent]{overpic}

\newcommand{\RR}{\mathbb{R}}

\newcommand{\eps}{\varepsilon}
\newcommand{\p}{{\bf p}}

\newtheorem{thm}{Theorem}%[section]
\newtheorem{lem}[thm]{Lemma}

\newtheorem{prop}[thm]{Proposition}

\theoremstyle{definition}

\theoremstyle{remark}

\begin{document}
\title{Expansion of trivariate polynomials using proximity}
\author{Orit E. Raz}
\author{
Orit E. Raz\thanks{Gen-Gurion University of the Negev, Beer Sheva, Israel. Email: oritraz@bgu.ac.il}}
\maketitle
\begin{abstract}
We extend the proximity technique of Solymosi and Zahl~\cite{SolZah} to the setting of trivariate polynomials.
In particular, we prove the following result: Let $f(x,y,z)=(x-y)^2+(\varphi(x)-z)^2,
$
where $\varphi(x)\in \RR[x]$ has degree at least 3. Then, for every finite $A,B,C\subset \RR$ each of size $n$, one has
$
|f(A,B,C)|=\Omega(n^{5/3-\eps}),
$
for every $\eps>0$, where the constant of proportionality depends on $\eps$ and on $\deg(\varphi)$. 
This improves the previous exponent $3/2$, due to Raz, Sharir, and De Zeeuw~\cite{RazShaDeZee4D}.
To the best of our knowledge, prior to this work no trivariate polynomial was known to have expansion exceeding $\Omega(n^{3/2})$.
\end{abstract}
\vspace{1ex}

\section{Introduction}
In many cases in combinatorial geometry, counting questions involving distances, slopes, collinearity, etc., can be reformulated as analogous counting questions involving grid points lying on certain algebraic varieties. A unified study of such problems began with a question of Elekes~\cite{Ele} about expansion of bivariate real polynomials $f(x,y)$. 
Specifically, he asked: For a bivariate polynomial $f\in \RR[x,y]$ and given finite sets $A,B\subset \RR$, how small can be the image set
$$
f(A,B)=\{f(a,b)\mid a\in A, b\in B\}.
$$
Elekes conjectured that the image of $f$ on an $n\times n$ Cartesian product must be of cardinality superlinear in $n$, unless $f$ has a very concrete {\it special form}. This was confirmed in 2000 by Elekes and R\'onyai~\cite{ER00}, who proved the following dichotomy: Either $f$ is one of the forms
\begin{align}
f(x,y)&=h(p(x)+q(y))\quad\text{or}\nonumber\\f(x,y)&=h(p(x)q(y)),\label{specialbivariate}
\end{align} for some univariate real polynomials $p,q,h$, or, otherwise, for every finite $A,B\subset \RR$, each of size $n$, we have
\begin{equation}\label{eq:ER}
|f(A,B)|=\omega(n).
\end{equation}

In \cite{RazShaSol}, Raz, Sharir and Solymosi introduced a new proof of the Elekes--R\'onyai theorem, which also yields improved bounds on the expansion of $f$. 
Roughly speaking, for a given $f\in \RR[x,y]$ and two finite sets $A,B\subset \RR$, they bounded the number of quadruples $((a,b),(a',b'))\in (A\times B)^2$ satisfying $f(a,b)=f(a',b')$ by reducing it to a  point-curve incidence problem in the plane. Concretely, incidences between the point set $A\times A$ and the family of curves 
$\{\gamma_{b,b'}\mid (b,b')\in B\times B\},$
where $\gamma_{b,b'}$ is given by the equation 
$$f(x,b)=f(y,b').$$ 
If these curves are distinct, known incidence bounds can be applied to obtain the desired estimate. However, it may happen that many of the curves coincide, in which case the incidence bound breaks down. Raz, Sharir, and Solymosi showed that this ``failure” occurs if and only if 
$f$ has a special form, namely one of the forms described in \eqref{specialbivariate}. 
In the non-special case, their analysis yields the lower bound $\Omega(n^{4/3})$

Solymosi and Zahl~\cite{SolZah} recently improved the expansion bound for $f$ in the non-special case, establishing the lower bound $\Omega(n^{3/2})$. Their argument builds on the framework of \cite{RazShaSol}, but applies the incidence bound to carefully selected subsets of points and curves. In particular, they restrict to points $(a,a')\in A\times A$ with $a$ and $a'$ sufficiently close, and, similarly, to curves $\gamma_{b,b'}$ with parameters $(b,b')\in B\times B$  with $b$ and $b'$ sufficiently close. This refinement, referred to as the {\it proximity method}, 
overcomes a loss incurred in the earlier argument from an application of the Cauchy--Schwarz inequality.

A $3$-variate analogue of the Elekes--R\'onyai theorem was studied in Raz, Sharir, and De Zeeuw~\cite{RazShaDeZee4D}. The case of 
  $k\ge 4$ variables is studied by Raz and Shem Tov in \cite{RazShe}.
\begin{thm}[{\bf \cite{RazShaDeZee4D, RazShe}}]\label{dvarER}
Let $k\ge 3$ and assume that  $f\in \RR[x_1,\ldots,x_k]$ depends non-trivially on each of its varaiables. Then one of the following holds:\\
(i) For every finite $A_1,\ldots,A_k\subset \RR$, with $|A_i|=n$, for $i=1,\ldots,k$, one has
$$
|f(A_1,\ldots,A_k)|=\Omega(n^{3/2}),
$$
where the constant of proportionality depends only on $\deg(f)$ and on $k$. \\
(ii) $f$ is of one of the special forms
\begin{align}
f(x_1,\ldots,x_k)&=h(p_1(x_1)+\cdots +p_k(x_k))\label{kvarspecial}\\
f(x_1,\ldots,x_k)&=h(p_1(x_1)\cdot\ldots\cdot p_k(x_k)) \nonumber
\end{align}
for some univariate real polynomials $p_1,\ldots,p_k,h$.
\end{thm} 

Similar to the bivariate case, the analysis for $k=3$ also reduces to a point-curve incidence problem in the plane. 
Specifically, for a trivariate polynomial $f\in \RR[x,y,z]$ and  finite sets $A,B,C\subset~\RR$, define $D:=f(A,B,C)$ and consider incidences between the point set $A\times D$ and the family of curves 
$
\{\gamma_{b,c}\mid (b,c)\in B\times C\},
$
where $\gamma_{b,c}$ is given by the equation 
$$
w=f(x,b,c).
$$

In some respects, the trivariate case is simpler than the bivariate one. Here, the structure of $f$ is more directly reflected in properties of  the variety 
$w=f(x,y,z)$, whereas in the bivariate case one must analyze the more intricate variety $f(x,y)=f(z,w)$. Moreover, in the trivariate setting the argument avoids
an application of the Cauchy--Schwarz inequality, therefore eliminating the loss in the estimate that arises in the bivariate case. 

As mentioned above, the proximity method of Solymosi and Zahl~\cite{SolZah} improves the bivariate bound precisely by overcoming this loss from the Cauchy--Schwarz step. For this reason, it is not immediately clear how their method could extend to the trivariate case.

%\subsection{Our results.} 
In this paper, we show how proximity can in fact be used to obtain a stronger expansion bound for trivariate polynomials. In particular, we establish an improved bound for a concrete family of trivariate polynomials. Namely, we prove the following main result.

\begin{thm}\label{main}
Let 
$$
f(x,y,z)=(x-y)^2+(\varphi(x)-z)^2,
$$
where $\varphi(x)$ is a univariate real polynomial of degree at least $3$. 
Then, for every $\eps>0$ and any finite sets $A,B,C\subset\RR$, each of size $n$, we have
$$
|f(A,B,C)|=\Omega(n^{5/3-\eps}),
$$
where the constant of proportionality depends on $\deg\varphi$ and on $\eps$. 
\end{thm}
This improves upon the previous bound of $\Omega(n^{3/2})$,
which follows from Theorem~\ref{dvarER}. 
To the best of our knowledge, prior to this work no trivariate polynomial was known to have expansion exceeding $\Omega(n^{3/2})$. 

We remark that our approach is more general 
and extends to other trivariate polynomials.
A natural direction for future research is to determine the precise subfamily of polynomials $\mathcal{F}\subset \mathbb{R}[x,y,z]$ for which the method yields the improved lower bound $\Omega(n^{5/3-\eps})$.

\section{Preliminaries}\label{sec:pre}

\subsection{Point-curve incidences in the plane}
For a finite set of points ${\cal P}\subset \RR^2$ and a finite set of planar curves ${\cal C}$, we let $I({\cal P}, {\cal C})$ denote the set of point-curve incidences; that is
$$
I({\cal P},{\cal C})=\{(p,\gamma)\in {\cal P}\times{\cal C}\mid p\in \gamma\}.
$$
The classical Szemer\'edi--Trotter theorem~\cite{SzeTro} asserts that, for the special case where ${\cal C}$ is a set of lines, and putting  $m:=|{\cal P}|$ and $n:=|{\cal C}|$, one has 
$$
\left|I({\cal P},{\cal C})\right|=O\left(m^{2/3}n^{2/3}+m+n\right).
$$

Since the Szemer\'edi--Trotter paper, numerous alternative proofs and related problems have been studied. For our result, we require an extension of Szemer\'edi--Trotter to point-curve incidence problems, established by Sharir and Zahl~\cite{ShaZah}, in which the curves are algebraic and form an 
$s$-dimensional family. 
We now recall the relevant definitions from their work.

A bivariate polynomial $h\in\RR[x,y]$ of degree at most $D$ is a linear combination of the form
$$
h(x,y)=\sum_{0\le i+j\le D}c_{ij}x^iy^j.
$$
Note that the number of monomials $x^iy^j$ such that $0\le i+j\le D$ is $\binom{D+2}{2}$.
In this sense, every point $\vec{c}\in \RR^{\binom{D+2}{2}}$ (other than the all-zero vector) can be associated with a curve in $\RR^2$, given by the zeroset of the bivariate polynomial whose coefficients are the entries of $\vec{c}$. 
If $\lambda\neq 0$, then $f$ and $\lambda f$ have the same zero-set.
Thus, the set
of algebraic curves that can be defined by a polynomial of degree at most $D$ in $\RR^2$ can
be identified with the points in the projective space 
${\bf P}\RR^{\binom{D+2}{2}}$.

Define an \emph{$s$-dimensional family of plane curves of degree at most $D$} to be  an algebraic variety $F\subset{\bf P}\mathbf \RR^{\binom{D+2}{2}}$ such that $\dim(F)=s$. We will call the degree of the variety $F$ the \emph{complexity of the family}. 

They then proved the following incidence bound:
\begin{thm}[{\bf Sharir--Zahl~\cite{ShaZah}}] \label{ShaZah}
Let $\Gamma$ be a set of $n$ algebraic plane curves that belong to an $s$-dimensional family of curves of degree at most $D$ of constant complexity at most $K$, no two of which
share a common irreducible component. 
Let $P$ be a set of $m$ points in the plane. Then
for any $\eps>0$, the number of incidences $|I(P,\Gamma)|$ between the points of $P$ and the curves
of $\Gamma$ satisfies
$$
|I(P,\Gamma)|=O_\eps\left( m^{\frac{2s}{5s-4}}n^{\frac{5s-6}{5s-4}+\eps}\right)+O\left(m^{2/3}n^{2/3}+m+n\right),
$$
where the constant of proportionality depends on $s$, $K$, $D$, and in the first term also on $\eps$.
\end{thm}

\subsection{Symmetry of curves}
Given a set $S\subset \RR^2$ and a transformation $T:\RR^2\to\RR^2$, we say that $T$ {\it fixes} $S$ if $T(S) = S$. 
We say that a transformation $T$ is a {\it symmetry} of a plane algebraic curve $C$ if $T$ is an isometry of $\RR^2$ and fixes $C$. Recall that an isometry of $\RR^2$ is either a rotation, a translation, or a glide reflection (a reflection combined with a translation).

The following lemma is proved by Pach and De Zeeuw~\cite[Lemma 2.5]{PachDeZeeuw}.
\begin{lem}[{\bf Pach and De Zeeuw~\cite[Lemma 2.5]{PachDeZeeuw}}]\label{lem:symmetries}
An irreducible plane algebraic curve of degree $d$ has at most $4d$ symmetries, unless it is a line or a circle.
\end{lem}

\subsection{Complete bipartite graphs}\label{subsec:bipartite}
We introduce some notation and recall some properties in the spirit of Raz--Solymosi~\cite[Section 3]{RazSol}. 
For completeness we give all the details here. 
Let $\p=(p_1,\ldots,p_k), \p'=(p_1',\ldots,p_k')\in \RR^2\times\cdots\times \RR^2\cong\RR^{2k} $ be two $k$-tuples of points in $\RR^2$. 
Define
$$
\Sigma_{\p,\p'}=\{(q,q')\in \RR^2\times \RR^2\mid \text{$\|p_i-q\|=\|p_i'-q'\|$ for each $i=1,\ldots,k$}\}\subset \RR^2\times \RR^2.
$$
Let $\sigma_{\p,\p'}$ (respectively, $\sigma'_{\p,\p'}$) denote the projection of $\Sigma_{\p,\p'}$ to the first (respectively, last) copy of $\RR^2$ in $\RR^2\times \RR^2$.

\begin{lem}\label{conic}
Let $\p=(p_1,p_2,p_3), \p'=(p_1',p_2',p_3')\in  (\RR^2)^3$, and assume that $p_1,p_2,p_3$ are pairwise distinct.
Then either\\
(i)  $\p$ and $\p'$ are congruent, or \\
(ii) $\Sigma_{\p,\p'}$ is one-dimensional, and each of $\sigma_{\p,\p'}$ and $\sigma'_{\p,\p'}$ is an algebraic curve of degree at most two.
\end{lem}
\begin{proof}
By definition, for $(q,q')\in \Sigma_{\p,\p'}$ we have
\begin{align*}
\|p_i-q\|^2=&\|p_i'-q'\|^2,~~~i=1,2,3
\end{align*}
or
$$
\|p_i\|^2-2p_i\cdot q+\|q\|^2=\|p_i'\|^2-2p_i'\cdot q'+\|q'\|^2,~~~i=1,2,3.
$$
Subtracting  the $3$rd equation from each of the first two equations, we get the system
\begin{align}
\|p_1\|^2-\|p_3\|^2-2(p_1-p_3)\cdot q&= \|p'_1\|^2-\|p'_3\|^2-2(p_1'-p_3')\nonumber\\
\|p_2\|^2-\|p_3\|^2-2(p_2-p_3)\cdot q&= \|p'_2\|^2-\|p'_3\|^2-2(p_1'-p_3')\label{sysK3n}\\
\|p_3\|^2-2p_3\cdot q+\|q\|^2&=\|p_3'\|^2-2p_3'\cdot q'+\|q'\|^2.\nonumber
\end{align}
The system can be rewritten as
\begin{align*}
\tfrac12u-A q&=\tfrac12 v-B q', \\
\|p_3\|^2-2p_3\cdot q+\|q\|^2&=\|p_3'\|^2-2p_3'\cdot q'+\|q'\|^2 ,
\end{align*}
where $A$ (resp., $B$) is a $2\times 2$ matrix whose $i$th row equals $p_i-p_3$ 
(resp., $p'_i-p'_3$), for $i=1,2$, and 
\begin{align*}
u & = \left(\begin{matrix}
           \|p_1\|^2-\|p_3\|^2 \\
          \|p_2\|^2-\|p_3\|^2 
         \end{matrix}
         \right) , \quad 
         v   = \left(\begin{matrix}
           \|p'_1\|^2-\|p'_3\|^2 \\
          \|p'_2\|^2-\|p'_3\|^2 
         \end{matrix}
         \right) 
\end{align*}
are vectors in $\RR^2$. 

Assume first that $p_1',p_2',p_3'$ are not collinear. So the matrix $B$ is invertible and we have 
$$
q'=B^{-1}Aq+w,
$$
for $w = \frac12 B^{-1}(v-u) \in\RR^2$. Let $T(q):=B^{-1}Aq+w$. 
So $(q,q')\in \Sigma_{\p,\p'}$ if and only if $q'=T(q)$. Plugging this is the 3rd equation in \eqref{sysK3n} we get
\begin{equation}\label{projq}
\|p_3\|^2-2p_3\cdot q+\|q\|^2=\|p_3'\|^2-2p_3'\cdot T(q)+\|T(q)\|^2,
\end{equation}
which defines $\sigma_{\p,\p'}$. Note that this gives a conic section unless \eqref{projq} is trivial, i.e., the zero equation. 
Note that for this to happen the quadratic part has to be zero. That is, 
$$
\langle q,q\rangle =\langle B^{-1}Aq,B^{-1}A q\rangle 
$$
%or
%$$
%\langle q,q\rangle =\langle q,(B^{-1}A)^{tr}B^{-1}A q\rangle 
%$$
or
$$
\langle q,\left(I-(B^{-1}A)^{tr}B^{-1}A\right)q\rangle =0
$$
This defines a trivial quadratic equation if and only if 
$$
I-(B^{-1}A)^{tr}B^{-1}A=0
$$
or
%$$
%(B^{-1}A)^{tr}B^{-1}A=I
%$$
%or 
$$
(B^{-1}A)^{tr}=(B^{-1}A)^{-1},
$$
%so 
%$$
%\text{$B^{-1}A$ is an orthogonal matrix},
%$$
which means that $T$ is an isometry of $\RR^2$. 
This implies that $\p$ and $\p'$ are congruent.
We conclude that $\sigma_{\p,\p'}$ is either conic section given by  \eqref{projq} or $\p$ and $\p'$ are congruent, which completes the proof for the case that $p_1',p_2',p_3'$ are non-collinear. 

By symmetry, same analysis applies also when $p_1,p_2,p_3$ are non-collinear. 
So we need to prove the lemma for the case that each of the triples $p_1,p_2,p_3$ and $p_1',p_2',p_3'$ is collinear. In this case we may assume without loss of generality that $p_3=p_3'=(0,0)$, and $p_1=(a,0)$, $p_1'=(a',0)$, $p_2=(b,0)$, and $p_2'=(b',0)$. Note that our assumption that $p_1,p_2,p_3$ are pairwise distinct implies $b\neq a\neq 0$. Writing $q=(x,y)$ and $q'=(x',y')$, the system of equations defining $\Sigma_{\p,\p'}$ in this case becomes
\begin{align*}
(x-a)^2+y^2&=(x'-a')^2+(y')^2\\
(x-b)^2+y^2&=(x'-b')^2+(y')^2\\
x^2+y^2&=(x')^2+(y')^2
\end{align*}
or
\begin{align*}
-2ax+a^2&=-2a'x'+(a')^2\\
-2bx+b^2&=-2b'x'+(b')^2\\
x^2+y^2&=(x')^2+(y')^2.
\end{align*}
Recalling our assumption that $p_1,p_2,p_3$ are distinct, we have $b\neq a\neq 0$. Thus the first two equations give 
\begin{align*}
x&=(a'/a)x'-(a')^2/(2a)+a/2\\
x&=(b'/b)x'-(b')^2/(2b)+b/2.
\end{align*}
So we either get unique values for $x,x'$ or otherwise $a'/a=b'/b$. In the former case both $\sigma_{\p,\p'}$ and $\sigma'_{\p,\p'}$  are lines parallel to the $y$-axis, and so in particular (ii) holds. 
In the latter case, write $a'/a=b'/b=t$. Assume first that $t\neq 0$. Then 
\begin{align*}
x&=tx'+(a/2)(1-t^2)\\
x&=tx'+(b/2)(1-t^2),
\end{align*}
and so either $t=\pm 1$, in which case $\p$ and $\p'$ are congruent, or $a=b$ which is a contradiction to our assumption that $p_1,p_2,p_3$ are distinct. 
So we may assume that $t=0$, which means that $a'=b'=0$. 
Assume $(q,q')\in \Sigma_{\p,\p'}$ and write $q=(x,y)$ and $q'=(x',y')$. 
Let $r:=(x')^2+(y')^2$. So $q=(x,y)$ must satisfy the system
\begin{align*}
(x-a)^2+y^2&=r\\
(x-b)^2+y^2&=r\\
x^2+y^2&=r.
\end{align*}
However, recalling that $b\neq a\neq 0$, this system has no solutions. So $\Sigma_{\p,\p'}$ is empty in this case. This completes the proof of the lemma. 
\end{proof}

\section{Proof of Theorem~\ref{main}}
Let $f$ be as in the statement and let $A,B,C\subset \RR$ be finite sets, each of size $n$. 
Set
$$
D:=f(A,B,C).$$ 
Our goal is to lower bound $|D|$. 

Let 
$$
t=n^{3/2}/(s|D|^{1/2}),$$ where $s>0$ is a sufficiently large constant, to be determined later. Consider the partition of each $A,B, C$ into $t$ consecutive segments, each containing at most $\lceil n/t\rceil $ elements.
Let $\{A_i\}_{i=1}^t$, $\{B_i\}_{i=1}^t$, $\{C_i\}_{i=1}^t$ stand for the corresponding partitions. 
We write $a\sim a'$ if $a\neq a'$ and there exists some index $i\in [t]$ such that $a,a'\in A_i$. We write similarly $b\sim b'$ and $c\sim c'$ according to the partitions of $B$ and $C$ respectively. Define
$$
Q:=\{((a,b,c), (a',b',c'))\in (A\times B\times C)^2\mid f(a,b,c)=f(a',b',c'),\;\;a\sim a',\;b\sim b',\;c\sim c'\}
$$

We prove a lower bound on $|Q|$. 
\begin{prop}\label{prop:lowerQ}
We have 
\begin{equation}\label{lowerQ}
|Q|=\Omega(sn^3),
\end{equation}
where the constant of proportionality depends only on $\deg(\varphi)$.
\end{prop}
\begin{proof}
For each $d\in D$ let 
$$
G_d:=\{(a,b,c)\in A\times B\times C\mid f(a,b,c)=d\}.
$$
This defines a partition 
$A\times B\times C=\bigcup_{d\in D}G_d,
$ 
and thus in particular
$$
\sum_{d\in D}\left|G_d\right|=n^3.
$$

Let 
$$
D':=\{d\in D\mid |G_d|\ge n^3/(10|D|)\}.
$$
Note that 
\begin{align*}
n^3
&=\sum_{d\in D}|G_d|\\
&=\sum_{d\in D'}|G_d|+\sum_{d\in D\setminus D'}|G_d|\\
&\le\sum_{d\in D'}|G_d|+n^3/(10|D|)\cdot |D|
\end{align*}
and so
\begin{equation}\label{sumDprime}
\sum_{d\in D'}|G_d|\ge (9/10)n^3.
\end{equation}

Next, fix $d\in D'$. 
Note that for each $(i,j,k)\in[t]^3$, the set $A_i\times B_j\times C_k$ is contained in an axis-parallel box in $\RR^3$. Moreover, the boundaries of those boxes lie in the union of $3t$ planes. 
Let $T_d\subset [t]^3$ denote the subset of 3-tuples of indices of boxes that intersect the surface $f(x,y,z)=d$. Note that $|T_d|=O(t^2)$, with constant of proportionality that depends only on $\deg(\varphi)$.
Let $T_d'\subset T_d$ be those 3-tuples of indices for which $|G_d\cap (A_i\times B_j\times C_k)|\ge s$.

We have
\begin{align*}
|G_d|
&=\sum_{(i,j,k)\in T_d}|G_d\cap (A_i\times B_j\times C_k)|\\
&=\sum_{(i,j,k)\in T_d'}|G_d\cap (A_i\times B_j\times C_k)|+\sum_{(i,j,k)\in T_d\setminus T_d'}|G_d\cap (A_i\times B_j\times C_k)|\\
&\le \sum_{(i,j,k)\in T_d'}|G_d\cap (A_i\times B_j\times C_k)|+s\cdot O(t^2)\\
&\le \sum_{(i,j,k)\in T_d'}|G_d\cap (A_i\times B_j\times C_k)|+O(n^3/(s|D|))\\
&\le \sum_{(i,j,k)\in T_d'}|G_d\cap (A_i\times B_j\times C_k)|+|G_d|/2,
\end{align*}
where the inequality on the fourth line is by our choice of $t$, 
and the inequality on the last line is because $d\in D'$ and assuming $s>0$ is taken sufficiently large.
Thus
\begin{equation}\label{eq:Gd}
\sum_{(i,j,k)\in T_d'}|G_d\cap (A_i\times B_j\times C_k)|\ge |G_d|/2.
\end{equation}

Using $\binom{m}{2}\ge cm$ for $m\ge 2c+1$, we conclude that, for every $d\in D'$, one has
\begin{align*}
\sum_{(i,j,k)\in T_d'}\left|\binom{G_d\cap (A_i\times B_j\times C_k)}{2}\right|
&\ge 
\frac{s-1}{2}\sum_{(i,j,k)\in T_d'}|G_d\cap (A_i\times B_j\times C_k)|\\
&\ge 
\frac{s-1}{4}|G_d|\\
&\ge
\frac{s}{5}|G_d|
,
\end{align*}
where the last inequality assumes $s\ge 5$. 
Thus we get
\begin{align*}
|Q|
&=\sum_{d\in D}\sum_{(i,j,k)\in [t]^3}\left|\binom{G_d \cap (A_i\times B_j\times C_k)}{2}\right|\\
&\ge \sum_{d\in D'}\sum_{(i,j,k)\in T_d'}\left|\binom{G_d \cap (A_i\times B_j\times C_k)}{2}\right|\\
&\ge \frac{s}{5} \sum_{d\in D'}|G_d|\\
&\ge (9/50)sn^3,
\end{align*}
where the last inequality uses \eqref{sumDprime}. 
This completes the proof of Proposition~\ref{prop:lowerQ}.
\end{proof}

Next we prove an upper bound on $|Q|$.
\begin{prop}\label{prop:upperQ}
For every $\eps>0$, we have
$$
|Q|\le O_\eps\left( (s^2n|D|)^{9/8+\eps}\right)+4\deg( \varphi) n^3.
$$
\end{prop}
\begin{proof}
We reduce the problem of upper bounding $Q$ to a point-curve incidence problem in the plane. 
For this, we associate with each $((b,c),(b',c'))\in  (B\times C)^2$ the planar curve (with coordinates $(x,x')$) given by the equation
$$
f(x,b,c)=f(x',b',c').
$$
Note that distinct choices of $((b,c),(b',c'))$ might give rise to the same curve, or to distinct curves that share an irreducible component. 
Let 
$$
\Gamma:=\{\text{$\gamma\subset \gamma_{b,c,b',c'}$ is irreducible }
\mid \text{$((b,c),(b',c'))\in (B\times C)^2$, $b\sim b'$, $c\sim c'$}\}.
$$
Consider the set of parameters associated with the curves, that is, the set
$$
\hat\Gamma:=\{((b,c),(b',c'))\in (B\times C)^2\mid \; b\sim b',\; c\sim c'\}.
$$
For each $\gamma\in \Gamma$ we define 
$$
m(\gamma)=\{((b,c),(b',c'))\in \hat\Gamma\mid \gamma\subset \gamma_{b,c,b',c'}\}
$$

We show that there exists a finite subset $\Gamma_0\subset \Gamma$ such that 
for every $\gamma\in\Gamma\setminus\Gamma_0$ we have $|m(\gamma)|\le 4$. 
Indeed, let $\gamma\in \Gamma$ and suppose that $|m(\gamma)|\ge 5$. 
So there are three pairs $p_i:=(b_i,c_i)$, $p_i':=(b_i',c_i')$, with $(p_i,p_i')\in m(\gamma)$ for $i=1,2,3$, such that without loss of generality $p_1,p_2,p_3$ are distinct.

Let $\p=(p_1,p_2,p_3)$ and $\p'=(p_1',p_2',p_3')$ and recall the definition of $\Sigma_{\p,\p'}$ from Section~\ref{subsec:bipartite}.
By definition of $m(\gamma)$, this means that for every $(x,x')\in \gamma$ one has
$$
((x,\varphi(x)), (x',\varphi(x'))\in \Sigma_{\p,\p'}.
$$
By Lemma~\ref{conic}, either $\p,\p'$ are congruent, or $\sigma_{\p,\p'}$ is a conic section. In the latter case this implies that $\{(t,\varphi(t))\mid t\in \RR\}$ is a conic section, which is a contradiction because $\deg(\varphi)\ge 3$ and this curve is irreducible. 
Thus necessarily $\p$ and $\p'$ are congruent. 
Thus, there exists an isometry of $\RR^2$, $R$, such that 
$$(x',\varphi(x'))=R(x,\varphi(x)),\text{ for every $(x,x')\in \gamma$.}
$$ 
Since $\gamma$ is infinite, this implies that $R$ maps $(t,\varphi(t))$ to itself.  
By Lemma~\ref{lem:symmetries}, we have that $R$ is one of at most $4\deg (\varphi)$ possible isometries, which in turn determines the corresponding curve $\gamma$ up to at most $4\deg (\varphi)$ possibilities. We conclude that
$$|\Gamma_0|\le 4\deg (\varphi)$$ and 
\begin{equation}
\text{for every 
$\gamma\in \Gamma\setminus\Gamma_0$ we have $|m(\gamma)|\le 4$.}
\end{equation}

Next, let
$$
\hat \Gamma_0:=\{((b,c),(b',c'))\in \hat\Gamma\mid \text{there exists $\gamma\in\Gamma_0$ such that $\gamma\subset \gamma_{b,c,b',c'}$}\}.
$$
We claim that 
\begin{equation}\label{eq:boundhatQ}
|\{((a,b,c),(a',b',c'))\in Q\mid ((b,c),(b',c'))\in \hat\Gamma_0\}|\le 4\deg (\varphi) n^3.
\end{equation}
Indeed, by what has just been argued, there exists an isometry $R$ of $\RR^2$, one of at most $4\deg (\varphi)$ possibilities, such that 
$(b',c')=R(b,c)$ and $(a',\varphi(a'))=R(a,\varphi(a))$. In other words, $(a',b',c')$ is determined by $(a,b,c)$, up to at most $4\deg (\varphi)$ possibilities.
Thus \eqref{eq:boundhatQ} follows.

Finally, let
$$
P=\{(a,a')\in A\times A\mid a\sim a'\},
$$
and apply Theorem~\ref{ShaZah} to bound to $|I(P,\Gamma\setminus\Gamma_0)|$. 
Note that 
$$
|P|=\Theta\left(\binom{n/t}{2}t\right)=\Theta(n^2/t)=\Theta(sn^{1/2}|D|^{1/2})
$$ 
and 
\begin{align*}
|\Gamma\setminus\Gamma_0|\le \deg(f)|\hat\Gamma|
&=\Theta\left(\binom{n/t}{2}\binom{n/t}{2}t^2\right)\\
&=\Theta(n^4/t^2)\\
&=\Theta(s^2n|D|).
\end{align*}

We get
\begin{align*}
|Q|
&\le
4|I(P,\Gamma\setminus\Gamma_0)|+
4\deg (\varphi) n^3,\\
&=O_\eps\left(
(sn^{1/2}|D|^{1/2})^{9/4+2\eps}\right)+
4\deg (\varphi) n^3.
\end{align*}
This proves Proposition~\ref{prop:upperQ}.
\end{proof}

Finally, combining the inequalities from Proposition~\ref{prop:lowerQ} and Proposition~\ref{prop:upperQ}, we conclude that
\begin{align*}
sn^3\le |Q| 
&=O_\eps\left(
(n|D|)^{9/8+\eps}\right)
+
4\deg (\varphi) n^3.
\end{align*}
Choosing $s>8\deg(\varphi)$ and rearranging, we get
$$
|D|=\Omega_\eps(n^{5/3-\eps'}).$$
This completes the proof of Theorem~\ref{main}.\hfill$\square$

\end{document}